\documentclass[10pt,reqno]{amsart}
\usepackage{amsmath,amscd,amssymb}
\usepackage{url}
\usepackage{color}

\theoremstyle{plain}
   \newtheorem{theorem}{Theorem}[section]
   \newtheorem{proposition}[theorem]{Proposition}
   \newtheorem{lemma}[theorem]{Lemma}
   \newtheorem{corollary}[theorem]{Corollary}

\theoremstyle{definition}
   \newtheorem{definition}{Definition}[section]

   \newtheorem{example}{Example}[section] 
\theoremstyle{remark}
 \newtheorem{remark}{Remark}[section]

\newcommand{\R}{\mathbb{R}}
\newcommand{\Z}{\mathbb{Z}}
\newcommand{\B}{\mathcal{B}}
\newcommand{\F}{\mathcal{F}}

\def\newop#1{\expandafter\def\csname #1\endcsname{\mathop{\rm
#1}\nolimits}}

\newop{vol}
\newop{Vol}
\newop{conv}
\newop{Supp}
\newop{supp}
\newop{aw-height}
\newop{Nasc}
\newop{nasc}
\newop{comp}
\newop{Des}
\newop{des}
\newop{maxDes}
\newop{awheight}
\newop{lcm}
\newop{red}
\newop{link}

\begin{document}
\title{Simplices for Numeral Systems}

\author{Liam Solus}
\date{\today}
\address{Matematik, KTH, SE-100 44 Stockholm, Sweden}
\email{solus@kth.se}

\begin{abstract}
The family of lattice simplices in $\R^n$ formed by the convex hull of the standard basis vectors together with a weakly decreasing vector of negative integers include simplices that play a central role in problems in enumerative algebraic geometry and mirror symmetry.  
From this perspective, it is useful to have formulae for their discrete volumes via Ehrhart $h^\ast$-polynomials.  
Here we show, via an association with numeral systems, that such simplices yield $h^\ast$-polynomials with properties that are also desirable from a combinatorial perspective.      
First, we identify $n$-simplices in this family that associate via their normalized volume to the $n^{th}$ place value of a positional numeral system.  
We then observe that their $h^\ast$-polynomials admit combinatorial formula via descent-like statistics on the numeral strings encoding the nonnegative integers within the system.  
With these methods, we recover ubiquitous $h^\ast$-polynomials including the Eulerian polynomials and the binomial coefficients arising from the factoradic and binary numeral systems, respectively.  
We generalize the binary case to base-$r$ numeral systems for all $r\geq2$, and prove that the associated $h^\ast$-polynomials are real-rooted and unimodal for $r\geq2$ and $n\geq1$.  
\end{abstract}

\maketitle
\thispagestyle{empty}

\section{Introduction}
\label{sec: introduction}
We consider the family of simplices defined by the convex hull
$$
\Delta_{(1,q)} :=\conv(e_1,\ldots,e_n,-q)\subset\R^n,
$$
where $e_1,\ldots,e_n$ denote the standard basis vectors in $\R^n$, and $q = (q_1,q_2,\ldots,q_n)$ is any weakly increasing sequence of $n$ positive integers. 
A convex polytope $P\subset\R^n$ with vertices in $\Z^n$ (i.e. a \emph{lattice polytope}) is called \emph{reflexive} if its polar body
$$
P^\circ:=\{y\in\R^n : y^Tx\leq1 \mbox{ for all $x\in P$}\}\subset\R^n
$$
is also a lattice polytope.    
In the case of $\Delta_{(1,q)}$, the reflexivity condition is equivalent to the arithmetic condition 
$$
q_i \mid 1+\sum_{j\neq i}q_j 
\qquad
\mbox{ for all }j\in[n], 
$$
which was used to classify all reflexive $n$-simplices in \cite{C02}.  
Reflexive simplices of the form $\Delta_{(1,q)}$ have been studied extensively from an algebro-geometric perspective since they exhibit connections with string theory that yield important results in enumerative geometry \cite{CK99}.  
These connections result in an explicit formula for all Gromov-Witten invariants which, in turn, encode the number of rational curves of degree $d$ on the quintic threefold \cite{CXGP91}.  
This observation lead to the development of the mathematical theory of \emph{mirror symmetry} in which reflexive polytopes and their polar bodies play a central role.  
In particular, the results of \cite{CXGP91} are intimately tied to the fact that $\Delta_{(1,q)}$ for $q = (1,1,1,1)\in\R^4$ contains significantly less lattice points (i.e. points in $\Z^n$) than its polar body.

As a result of this connection, the lattice combinatorics and volume of the simplices $\Delta_{(1,q)}$ have been studied from the perspective of geometric combinatorics in terms of their (Ehrhart) $h^\ast$-polynomials.  
The \emph{Ehrhart function} of a $d$-dimensional lattice polytope $P$ is the function $i(P;t):=|tP\cap\Z^n|$, where $tP:=\{tp:p\in P\}$ denotes the $t^{th}$ dilate of the polytope $P$ for $t\in\Z_{\geq0}$.  
It is well-known \cite{E62} that $i(P;t)$ is a polynomial in $t$ of degree $d$, and the \emph{Ehrhart series} of $P$ is the rational function
$$
\sum_{t\geq0}i(P;t)z^t = \frac{h_0^*+h_1^*z+\cdots+h_d^*z^d}{(1-z)^{d+1}},
$$
where the coefficients $h^*_0,h^*_1,\ldots,h^*_d$ are all nonnegative integers \cite{S80}.  
The polynomial $h^*(P;z):=h_0^*+h_1^*z+\cdots+h_d^*z^d$ is called the \emph{$h^*$-polynomial} of $P$.
The $h^\ast$-polynomial of $P$ encodes the typical Euclidean volume of $P$ in the sense that $d!\vol(P) = h^\ast(P;1)$.  
It also encodes the number of lattice points in $P$ since $h_1^\ast = |P\cap\Z^n|-d-1$ (see, for instance, \cite{BR07}).  

In addition to combinatorial interpretations of the coefficients $h_0^\ast,\ldots,h_d^\ast\in\Z_{\geq0}$, distributional properties of these coefficients is a popular research topic.  
Let $p:= a_0+a_1z+\cdots+a_dz^d$ be a polynomial with nonnegative integer coefficients.  The polynomial $p$ is called \emph{symmetric} if $a_i = a_{d-i}$ for all $i\in[d]$, it is called \emph{unimodal} if there exists an index $j$ such that $a_i\leq a_{i+1}$ for all $i<j$ and $a_{i}\geq a_{i+1}$ for all $i\geq j$, it is called \emph{log-concave} if $a_i^2\geq a_{i-1}a_{i+1}$ for all $i\in[d]$, and it is called \emph{real-rooted} if all of its roots are real numbers.  
An important result in Ehrhart theory states that $P\subset\R^n$ is reflexive if and only if $h^\ast(P;x)$ is symmetric of degree $n$ \cite{H92}.   
It is well-known that if $p$ is real-rooted then it is log-concave, and if all $a_i>0$ then it is also unimodal.  
Unlike symmetry, there is no general characterization of any one of these properties for $h^\ast$-polynomials.  

The distributional (and related algebraic) properties of the $h^\ast$-polynomials for the simplices $\Delta_{(1,q)}$  were recently studied in terms of the arithmetic structure of $q$ \cite{BD16,BDS16}.  
In particular, \cite[Theorem 2.5]{BDS16} provides an arithmetic formula for $h^\ast(\Delta_{(1,q)};z)$ in terms of $q$ which the authors use to prove unimodality of the $h^\ast$-polynomial in some special cases.  
On the other hand, the literature lacks examples of $\Delta_{(1,q)}$ admitting combinatorial formula for their $h^\ast$-polynomials.  
Thus, while the simplices $\Delta_{(1,q)}$ constitute a class of convex polytopes fundamental  in algebraic geometry, we would like to observe that they are of interest from a combinatorial perspective as well.  
To demonstrate that this is in fact the case, we will identify simplices $\Delta_{(1,q)}$ whose $h^\ast$-polynomials are well-known in combinatorics and admit the desirable distributional properties mentioned above.  
We observe that simplices within this family can be associated in a natural way to \emph{numeral systems} and the $h^\ast$-polynomials of such simplices admit a combinatorial interpretation in terms of descent-type statistics on the numeral representations of nonnegative integers within these systems.  
Most notably, we find such ubiquitous generating polynomials as the \emph{Eulerian polynomials} and the \emph{binomial coefficients} arising from the \emph{factoradic} and \emph{binary} numeral systems, respectively.  
For these examples, we also find that the combinatorial interpretation of $h^\ast(\Delta_{(1,q)};z)$ is closely tied to that of $q$.  
We then generalize the simplices arising from the binary system to simplices associated to any base-$r$ numeral system for $r\geq2$, all of whose $h^\ast$-polynomials admit a combinatorial interpretation in terms of the base-$r$ representations of the nonnegative integers.  
Finally, we show that, even though these simplices are not all reflexive, their $h^\ast$-polynomials are always real-rooted, log-concave, and unimodal.  

The remainder of the paper is outlined as follows.
In Section~\ref{sec: numeral systems}, we review the basics of positional numeral systems and outline our notation.  
In Section~\ref{sec: reflexive numeral systems}, we describe when a numeral system associates to a family of reflexive simplices of the form $\Delta_{(1,q)}$, one in each dimension $n\geq1$.  
We then show that the $n$-simplex for the factoradic numeral system has the $(n+1)^{st}$ Eulerian polynomial as its $h^\ast$-polynomial.    
We also prove the $n$-simplex for the binary numeral system has $h^\ast$-polynomial $(1+x)^n$.  
In Section~\ref{sec: the positional base r numeral systems}, we generalize the binary case to base-$r$ numeral systems for $r\geq2$.  
We provide a combinatorial formula for these $h^\ast$-polynomials and prove they are real-rooted, log-concave, and unimodal for $r\geq2$ and $n\geq1$.  

\section{Numeral Systems}
\label{sec: numeral systems}
The typical (positional) numeral system is a method for expressing numbers that can be described as follows.  
A \emph{numeral} is a sequence of nonnegative integers $\eta := \eta_{n-1}\eta_{n-2}\ldots\eta_1\eta_0$, and the location of $\eta_i$ in the string is called \emph{place $i$}.  
The \emph{digits} are the numbers allowable in each place, and the $i^{th}$ \emph{radix (or base)} is the number of digits allowed in place $i$.  
A \emph{numeral system} is a sequence of positive integers $a = (a_n)_{n=0}^\infty$ satisfying
$
a_0:=1<a_1<a_2<\cdots,
$
which we call \emph{place values}.  
To see why $a$ yields a system for expressing nonnegative integers, let $b\in\Z_{\geq0}$ and let $n$ be the smallest integer so that $a_n>b$.  
Dividing $b$ by $a_{n-1}$ and iterating gives
\begin{align*}
b &= q_{n-1}a_{n-1} + r_{n-1}           	&  0&\leq r_{n-1}<a_{n-1}  \\
r_{n-1} &= q_{n-2}a_{n-2} + r_{n-2}         	&  0&\leq r_{n-2}<a_{n-2} 	\\
\vdots					       	 	&  &\vdots			 	\\
r_{i+1} &= q_{i}a_{i} + r_{i}         	&  0&\leq r_{i}<a_{i} 	\\
\vdots				         		&  &\vdots			 	\\
r_2 &= q_{1}a_{1} + r_{1}         		&  0&\leq r_{2}<a_{2} 	\\
r_1 &= q_{0}a_{0}.  					&     
\end{align*}
Denoting $b_a(i):=q_i$ for all $i$, it follows that 
\begin{equation}
\label{eqn: 0}
b = b_a(n-1)a_{n-1}+b_a(n-2)a_{n-2}+\cdots+b_a(1)a_{1}+b_a(0)a_{0}.  
\end{equation}
On the other hand, if $b = \sum_{i = 0}^{n-1}\beta_ia_i$ where $\sum_{j=0}^i\beta_ia_i < a_{i+1}$ for every $i\geq0$, then $b_a(i)= \beta_{i}$ for every $i$.  
In particular, the representation of $b$ in equation (\ref{eqn: 0}) is unique (see for instance \cite[Theorem 1]{F85}).  
Thus, the representation of the nonnegative integer $b$ in the numeral system $a$ is the numeral
$$
b_a:=b_a(n-1)b_a(n-2)\cdots b_a(1)b_a(0).
$$
An important family of numeral systems we will consider are the mixed radix systems.  
A numeral system $a = (a_n)_{n=0}^\infty$ is \emph{mixed radix} if there exists a sequence of integers $(c_n)_{n=0}^\infty$ with $c_0:=1$ and $c_n>1$ for $n>1$ such that the product
$
c_0c_1\cdots c_n = a_n
$
for all $n\geq 0$.  
In this case, $c_{n+1}$ is the $n^{th}$ radix of the system $a$.  

\begin{example}[Numeral systems]
\label{ex: numeral systems}
The following are examples of numeral systems.  

\begin{enumerate}
	\item The binary numbers are the numeral system $a = (2^n)_{n=0}^\infty$.  
	The radix of place $i$ is $2$ for every $i$ since each place assumes only digits $0$ or $1$.  
	The binary system is mixed radix with sequence of radices $c = (1,2,2,2,\ldots)$.  
	The numeral representation of the number $102$ in this system is $1100110$.  
	
	\item The ternary (base-$3$) numeral system is the system $a = (3^n)_{n=0}^\infty$.  
	The radix of place $i$ is $3$ for every $i$ since each place assumes only digits $0,1$ or $2$.  
	The ternary numeral system is also mixed radix, and it has sequence of radices $c = (1,3,3,3,3,\ldots)$.  
	The numeral representation of the number $102$ in the ternary system is $10210$.  
	
	\item An example of a numeral system that is not mixed radix is the system $a = (F_{n+1})_{n=0}^\infty$, where $a_n = F_{n+1}$ is the $(n+1)^{st}$ Fibonacci number.  
	This sequence is not mixed radix since there exist prime Fibonacci numbers other than $2$.  
	The radix of every place is $2$ since $2F_{n+1} = F_{n+1}+F_{n}>F_{n+1}$ for all $n$, and so each place assumes digits $0$ or $1$.  
	The numeral representation of the number $102$ in this system is $1000100000$.  
\end{enumerate}
\end{example}

\section{Reflexive Numeral Systems}
\label{sec: reflexive numeral systems}
In this section, we demonstrate a method for attaching an $n$-simplex of the form $\Delta_{(1,q)}$ to the $n^{th}$ place value $a_n$ of a numeral system $a$ such that $a_n = n!\vol(\Delta_{(1,q)})$, i.e. the \emph{normalized volume of $P$}.  
For certain families of numeral systems, which we call \emph{reflexive numeral systems}, these simplices can be chosen so that they are all reflexive.  
We show that the factoradic and binary numeral systems are reflexive, and recover the $h^\ast$-polynomials of their associated simplices.  
We also discuss the relationship between reflexive numeral systems and the mixed radix numeral systems, and the geometric relationship between the factoradic $n$-simplex and the $s$-lecture hall $n$-simplex with the same $h^\ast$-polynomial, i.e. the $n^{th}$ Eulerian polynomial.  
\begin{definition}
\label{def: reflexive numeral system}
A numeral system $a = (a_n)_{n=0}^\infty$ is called a \emph{reflexive (numeral) system} if there exists an increasing sequence of positive integers $d = (d_n)_{n=0}^\infty$ satisfying the following properties:
	\begin{enumerate}
		\item $d_i\mid a_n$ for all $0\leq i\leq n-1$, $n\geq 1$, and \label{condition 1}
		\item 
		$
		1+\sum_{i=0}^{n-1}\frac{a_{n}}{d_i} = a_n
		$
		for all $n\geq 1$.  \label{condition 2}
	\end{enumerate}
Such a sequence $d$ is called a \emph{divisor system (for a)}.
\end{definition}

\begin{example}[A reflexive numeral system]
\label{ex: reflexive system}
Recall from Example~\ref{ex: numeral systems} that the binary numeral system is given by $a = (a_n)_{n=0}^\infty := (2^n)_{n=0}^\infty$.  
The binary system $a$ admits the divisor system $d = (d_n)_{n=0}^\infty := (2^{n+1})_{n=0}^\infty$.  
This is because for all $n\geq1$ we have that $2^{i+1}\mid 2^n$ for all $0\leq i\leq n-1$, and 
$$
1+\sum_{i=0}^{n-1}\frac{a_{n}}{d_i} = 1+ \sum_{i=0}^{n-1}\frac{2^n}{2^{i+1}} = 1+ \sum_{i=0}^{n-1}2^{i} = 2^n = a_n.
$$
Thus, the binary numeral system $a$ is a reflexive system with divisor system $d$.  
Furthermore, in Theorem~\ref{thm: binomial coefficients} we will prove that 
$
h^\ast(\Delta_{(1,q)};z) = (1+z)^n.
$  
\end{example}

Recall that we would like to associate an $n$-simplex $\Delta_{(1,q)}$ to the $n^{th}$ place value $a_n$ of a numeral system $a$ by the relation $a_n = n!\vol(\Delta_{(1,q)})$.  
In \cite{BDS16, N07} is it shown that the normalized volume of $\Delta_{(1,q)}$ (i.e. the value $n!\vol(\Delta_{(1,q)})$) is $1+q_1+\cdots+q_n$.  
When $a$ is a reflexive system with divisor system $d$, condition (1) tells us that $\frac{a_n}{d_i}$ is a positive integer for all $n\geq1$ and $0\leq i\leq n-1$, and condition (2) tells us that the $n$-simplex $\Delta_{(1,q)}$ with 
$$
q:=\left( 
\frac{a_n}{d_{n-1}},\frac{a_n}{d_{n-2}},\ldots,\frac{a_n}{d_{0}}
\right)
\in\Z_{>0}^n
$$
will have the desired normalized volume.  
Moreover, it turns out the $h^\ast$-polynomial of this $n$-simplex can be computed in a recursive fashion in terms of the numeral representations of the first $a_n$ nonnegative integers.  
This recursive formula is presented in Proposition~\ref{prop: recursion for reflexive numeral systems}, but it is essentially due to the following theorem of \cite{BDS16}.  
\begin{theorem}
\cite[Theorem 2.5]{BDS16}
\label{thm: bds}
The $h^\ast$-polynomial for the $n$-simplex $\Delta_{(1,q)}$ for $q = (q_1,\ldots,q_n)$ is
$$
h^\ast(\Delta_{(1,q)};z) = 
\sum_{b=0}^{q_1+q_2+\cdots+q_n}z^{\omega(b)},
$$
where
$$
\omega(b) = b - \sum_{i=1}^n\left\lfloor\frac{q_ib}{1+q_1+q_2+\cdots+q_n}\right\rfloor.
$$
\end{theorem}

\begin{proposition}
\label{prop: recursion for reflexive numeral systems}
Let $a = (a_n)_{n=0}^\infty$ be a reflexive system that admits a divisor system $d = (d_n)_{n=0}^\infty$.  
Then for 
$$
q:=\left( 
\frac{a_n}{d_{n-1}},\frac{a_n}{d_{n-2}},\ldots,\frac{a_n}{d_{0}}
\right)
$$
the reflexive $n$-simplex $\Delta_{(1,q)}\subset\R^n$ has $h^\ast$-polynomial
$
h^\ast(\Delta_{(1,q)};z) = \sum_{b=0}^{a_n-1}z^{\omega(b)},
$
where 
$$
\omega(b) = \omega(b^\prime) + b_a(n-1) - \left\lfloor \frac{b}{d_{n-1}}\right\rfloor,
$$
and $b^\prime := b - b_a(n-1)a_{n-1}$.  
\end{proposition}

\begin{proof}
\phantom\qedhere
Applying Theorem~\ref{thm: bds} to our particular case, we can simplify the formula for $\omega(b)$ as follows:
\begin{equation*}
\begin{split}
\omega(b)
&=  b - \sum_{i=0}^{n-1}\left\lfloor\frac{b}{d_i}\right\rfloor,	\\
&=  b_a(n-1)a_{n-1}+b^\prime - \sum_{i=0}^{n-2}\left\lfloor\frac{b_a(n-1)a_{n-1}+b^\prime}{d_i}\right\rfloor-\left\lfloor\frac{b}{d_{n-1}}\right\rfloor,	\\
&=  b_a(n-1)a_{n-1}+b^\prime - b_a(n-1)\left(\sum_{i=0}^{n-2}\frac{a_{n-1}}{d_i}\right) - \sum_{i=0}^{n-2}\left\lfloor\frac{b^\prime}{d_i}\right\rfloor-\left\lfloor\frac{b}{d_{n-1}}\right\rfloor,	\\
&=  \omega(b^\prime)+b_a(n-1)-\left\lfloor\frac{b}{d_{n-1}}\right\rfloor. 	\mbox{\hspace{2.65in}$\square$}\\
\end{split}
\end{equation*}
\end{proof}


We note that not every mixed radix system is a reflexive system.  
However, if a mixed radix system is reflexive, then its corresponding divisor system is unique.
If $a$ is mixed radix then $a_n = c_0\cdots c_n$ where $c = (c_n)_{n=0}^\infty$ is its sequence of radices.  
Using this fact and part (2) of Definition~\ref{def: reflexive numeral system}, we deduce the following proposition.
\begin{proposition}
\label{prop: unique divisor system for reflexive mixed radix systems}
If $a = (a_n)_{n=0}^\infty$ is a reflexive mixed radix system with sequence of radices $c = (c_n)_{n=0}^\infty$ then its corresponding divisor system $d$ is unique and 
$$
d = (d_n)_{n=0}^\infty = \left(\frac{a_{n+1}}{c_{n+1}-1}\right)_{n=0}^\infty.  
$$
\end{proposition}


\begin{example}[Reflexive and non-reflexive mixed radix systems]
\label{ex: reflexive and non-reflexive mixed radix systems}
Recall from Example~\ref{ex: reflexive system} that the binary numeral system $a=(2^n)_{n=0}^\infty$ is a reflexive system with divisor system $d = (2^{n+1})_{n=0}^\infty$.  
Since $a$ is also a mixed radix system with sequence of radices $c = (1,2,2,2,2,\ldots)$, then we can apply Proposition~\ref{prop: unique divisor system for reflexive mixed radix systems} to recover the divisor system $d$ as the unique divisor system for $a$.  

It is also important to notice that there are mixed radix systems that are not reflexive.  
For example, consider the numeral system $a = (2^n\cdot n!)_{n=0}^\infty$ whose $n^{th}$ place value is the order of the hyperoctahedral group.  
Then $a$ is reflexive with sequence of radices $c = (1,2,4,\ldots,2n,\ldots)$.  
However, it follows from Proposition~\ref{prop: unique divisor system for reflexive mixed radix systems} that $a$ has no divisor system since there are infinitely many times when $\frac{2^{n+1}(n+1)!}{2(n+1)-1}$ is not an integer.  
For instance, every prime is expressible as $2(n+1)-1$ for some $n$, and this prime will never be a factor of $2^{n+1}(n+1)!$.  
\end{example}

Although identifying a divisor system for a given numeral system is generally a nontrivial problem, Proposition~\ref{prop: unique divisor system for reflexive mixed radix systems} makes it easier in the case of mixed radix systems, and it provides a quick check to deduce if the system is reflexive.  
Moreover, when a mixed radix system is reflexive, the resulting simplices $\Delta_{(1,q)}$ appear to have well-known $h^\ast$-polynomials with a combinatorial interpretation closely related to a combinatorial interpretation of $q$.
We now give two examples of this phenomenon.

\subsection{The factoradics and the Eulerian polynomials}
\label{subsec: the factoradics and the eulerian polynomials}

In this subsection, we study the numeral system $a = (a_n)_{n=0}^\infty := ((n+1)!)_{n=0}^\infty$, which is commonly called the \emph{factoradics}.  
By Proposition~\ref{prop: unique divisor system for reflexive mixed radix systems}, we see that the factoradics are reflexive with divisor system $d = (d_n)_{n=0}^\infty := ((n+1)!+n!)_{n=0}^\infty$.
We will see that the $q$-vectors given by $d$ admit a combinatorial interpretation in terms of descents, and the resulting simplices $\Delta_{(1,q)}$ have $h^\ast$-polynomials the Eulerian polynomials.  
For $\pi\in S_n$, we let
\begin{equation*}
\begin{split}
\Des(\pi) &:= \{i\in[n-1] : \pi_{i+1}>\pi_i\},	\\
\des(\pi) &:= \left|\Des(\pi)\right|,	\\
\maxDes(\pi) &:= 
\begin{cases}
\max\{i\in[n-1] : i\in\Des(\pi)\}	&	\mbox{for $\pi\neq 12\cdots n$,}	\\
0			&	\mbox{for $\pi = 12\cdots n$.}		\\
\end{cases}	\\
\end{split}
\end{equation*}
We then consider the pair of polynomial generating functions
$$
A_n(z) := \sum_{\pi\in S_n}z^{\des(\pi)} 
\quad 
\mbox{and}
\quad
B_n(z) := \sum_{\pi\in S_n}z^{\maxDes(\pi)},
$$
where $A_n(z)$ is the well-known \emph{$n^{th}$ Eulerian polynomial}.  
The polynomial $B_n(z)$ is a different generating function for permutations that satisfies a recursion described in Lemma~\ref{lem: max descents}.  
In the following, we let $A(n,k)$ and $B(n,k)$ denote the coefficient of $z^k$ in $A_n(z)$ and $B_n(z)$, respectively.  
\begin{lemma}
\label{lem: max descents}
For each $n\in\Z_{>0}$ we have that 
$$
B(n,0) = 1,
\quad 
 B(n,1) = n-1, 
$$
and for $k>1$
$$
B(n,k) = (n)_{k-1}(n-k) = nB(n-1,k-1).
$$
Moreover, for the polynomials $B_n(z)$, we have the recursive expression
$$
B_n(z) = 1-z+nzB_{n-1}(z), \qquad (n>1)
$$
and the closed-form expression 
$$
B_{n}(z) =1 + \sum_{k =1}^{n-1}\frac{n!}{(n-k)!+(n-k-1)!}z^k.
$$
\end{lemma}

\begin{proof}
By the definition of $B_n(z)$ we can see that $B(n,0) = 1$ and $B(n,1) = n-1$ for all $n\in\Z_{>0}$.  
To see that $B(n,k) = (n)_{k-1}(n-k)$ for $k>1$, notice first that the falling factorial $(n)_{k-1}$ counts the number of ways to pick the first $k-1$ letters of $\pi = \pi_1\pi_2\cdots\pi_n$.  
Multiplying this value by $(n-k)$ accounts for the fact that there are $(n-k)$ remaining choices for the $k^{th}$ letter such that $\pi_k>\pi_{k+1}$.  
The remaining $(n-k)$ letters of the permutation are then arranged in increasing order, and so $\maxDes(\pi) = k$.  
The recursion for $B(n,k)$ for $k>1$ and $B_n(z)$ for $n>1$ now follow readily from these observations.  
The closed form expression for $B_n(z)$ is immediate from the closed forms presented for the coefficients $B(n,k)$ for $k\geq 1$ and the identity $n! = (n-1)((n-1)!+(n-2)!)$.  
\end{proof}

We now show that the sequence of coefficients for the nonconstant terms of $B_n(z)$ is a vector $q$ for which the $n$-simplex $\Delta_{(1,q)}$ has $h^\ast$-polynomial $A_n(z)$; thereby offering, in a sense, a geometric transformation between the two generating polynomials.  

Recall that for two integer strings $\eta := \eta_1\eta_2\cdots\eta_n$ and $\mu := \mu_1\mu_2\cdots\mu_n$ the string $\eta$ is \emph{lexicographically larger} than the string $\mu$ if and only if the leftmost nonzero number in the string $(\eta_1-\mu_1)(\eta_2-\mu_2)\cdots(\eta_n-\mu_n)$ is positive.  
For $0\leq b <n!$, the factoradic representation of $b$, denoted
$
b_! := b_!(n-1)b_!(n-2)\cdots b_!(1)b_!(0),
$
is known as the \emph{Lehmer code} of the $b^{th}$ largest permutation of $[n]$ under the lexicographic ordering \cite{K98}.  
In particular, if we let $\pi^{(b)}$ denote the $b^{th}$ largest permutation of $[n]$ under the lexicographic ordering, then for all $0\leq i<n$
$$
b_!(i) = \left|\{0\leq j<i : \pi^{(b)}_{n-i}>\pi^{(b)}_{n-j}\}\right|.
$$
It is straightforward to check that $b_!(i)>b_!(i+1)$ if and only if $n-i\in\Des(\pi^{(b)})$.  
Thus, counting descents in $\pi^{(b)}$ is equivalent to counting descents in the factoradic representation of the integer $b$.  
This fact allows for the following theorem.  
\begin{theorem}
\label{thm: eulerian polynomials}
The factoradic numeral system $a = (a_n)_{n=0}^\infty = ((n+1)!)_{n=0}^\infty$ admits the divisor system $d = (d_n)_{n=0}^\infty = ((n+1)!+n!)_{n=0}^\infty$ for which the reflexive simplex $\Delta_{(1,q)}\subset\R^n$ with
$$
q:=\left( 
B(n+1,1),B(n+1,2),\ldots,B(n+1,n)
\right)
$$
has $h^\ast$-polynomial
$$
h^\ast(\Delta_{(1,q)};z) = A_{n+1}(z).
$$
\end{theorem}

\begin{proof}
First notice that $d = ((n+1)!+n!)_{n=0}^\infty$ is in fact a divisor system for $a = ((n+1)!)_{n=0}^\infty$ by Lemma~\ref{lem: max descents}.  
We now show, via induction on $n$, that for all integers $b = 0,1,\ldots,n!-1$
$
\omega(b) = \des(\pi^{(b)}),
$
where $\pi^{(b)}$ denotes the $b^{th}$ largest permutation of $[n]$.  
Since the base case $(n=1)$ is clear, we assume the result holds for $n-1$.  
For convenience, we reindex $a = (n!)_{n=1}^\infty$ and $d = (n!+(n-1)!)_{n=1}^\infty$.  
Then by Proposition~\ref{prop: recursion for reflexive numeral systems} we know that for $b = 0,1,\ldots,n!$
\begin{align}
\omega(b) &= \omega(b^\prime)+b_!(n-1)-\left\lfloor\frac{b}{d_{n-1}}\right\rfloor, \nonumber	\\
&= \omega(b^\prime)+b_!(n-1)-\left\lfloor\frac{(n-1)(b_!(n-1)(n-1)!+b^\prime}{n!}\right\rfloor, \nonumber	\\
&= \omega(b^\prime)+b_!(n-1)-\left\lfloor b_!(n-1)-\left(\frac{b_!(n-1)(n-1)!-(n-1)b^\prime}{n!}\right)\right\rfloor, \nonumber	\\
&= \omega(b^\prime)+\left\lceil\frac{b_!(n-1)(n-1)!-(n-1)b^\prime}{n!}\right\rceil. \label{eqn: 2}
\end{align}
With equation (\ref{eqn: 2}) in hand, we now consider a few cases.
First, notice that if $b_!(n-1)=0$ then $\omega(b) = \omega(b^\prime)$, and the result follows from the inductive hypothesis.  
This is because $0\leq b_{n-1}\leq (n-1)!-1$.  
So whenever $b_!(n-1) = 0$ then
$$
\left\lceil\frac{b_!(n-1)(n-1)!-(n-1)b^\prime}{n!}\right\rceil = \left\lceil-\frac{(n-1)b^\prime}{n!}\right\rceil=0.
$$
Suppose now that $0<b_!(n-1)< n$.  
Then $\pi^{(b)}$ satisfies $\pi_1^{(b)}=b_!(n-1)+1$, and we consider the following three cases.

First, if $b^\prime = 0$, then since $0<b_!(n-1)<n$, we have that $\omega(b) = \omega(b^\prime)+1$, and the result follows from the inductive hypothesis.  
Second, suppose that $b^\prime\neq0$ and that $0<b_!(n-1)(n-1)!-(n-1)b^\prime$.  
Then since $0<(n-1)b^\prime\leq b_!(n-1)(n-1)!<n!$, we know that 
$$
\left\lceil\frac{b_!(n-1)(n-1)!-(n-1)b^\prime}{n!}\right\rceil = 1.
$$
Thus, $\omega(b) = \omega(b^\prime)+1$.  
Since the first $b_!(n-1)(n-2)!$ permutations of $[n]$ with $\pi_1 = b_!(n-1)+1$ satisfy $\pi_1>\pi_2$, the result follows from the inductive hypothesis.  
Finally, if $b^\prime\neq0$ and $0\leq(n-1)b^\prime-b_1(n-1)(n-1)!$.  
Since $0<(n-1)b^\prime\leq(n-1)((n-1)!-1)<n!$, we have that 
$$
\left\lceil\frac{b_!(n-1)(n-1)!-(n-1)b^\prime}{n!}\right\rceil = -\left\lfloor\frac{(n-1)b^\prime-b_!(n-1)(n-1)!}{n!}\right\rfloor = 0.
$$
Thus, $\omega(b) = \omega(b^\prime)$.  
Since the last $(n-1)!-b_!(n-1)(n-2)!$ permutations of $[n]$ satisfying $\pi_1 = b_!(n-1)+1$ satisfy $\pi_1<\pi_2$, the result follows from the inductive hypothesis, completing the proof.  
\end{proof}

\begin{remark}[Relationship with lecture hall simplices]
\label{projective spaces}
Let $\F_n$ denote the $n$-simplex described in Theorem~\ref{thm: eulerian polynomials}.  
To the best of the author's knowledge, the only reflexive $n$-simplex $\Delta$ with $h^\ast(\Delta;z) = A_{n+1}(z)$, other than $\F_n$, is the \emph{$s$-lecture hall simplex}
$$
P_n^{(2,3,\ldots,n+1)}:=\left\{x\in\R^n : 0\leq\frac{x_1}{2}\leq\frac{x_2}{3}\leq\cdots\leq\frac{x_n}{n+1}\leq 1\right\}.
$$
Using the classification of \cite{C02}, we can deduce that $\F_n$ and $P_n^{(2,3,\ldots,n+1)}$ define distinct toric varieties in the following sense:
For a lattice $n$-simplex $\Delta\subset\R^n$ containing the origin in its interior, \cite[Definition 2.3]{C02} assigns a \emph{weight} $q:=(q_0,\ldots,q_n)\in\Z_{>0}^{n+1}$ and a \emph{factor} $\lambda:=\gcd(q_0,\ldots,q_n)$.   
We say that $\Delta$ is \emph{of type} $(q_{\red},\lambda)$, where $q_{\red}:=\frac{1}{\lambda}q$.  
When $\lambda = 1$, the toric variety of $\Delta$ is the weighted projective space $\mathbb{P}(q_{\red})$, and when $\lambda>1$, it is a quotient of $\mathbb{P}(q_{\red})$ by the action of a finite group of index $\lambda$.  

It follows from our construction that $\F_n$ has factor $1$, and so its toric variety is the weighted projective space $\mathbb{P}(B(n))$, where $B(n):=(1,B(n+1,1),\ldots,B(n+1,n))$.  
On the other hand, for $n>2$, $P_n^{(2,3,\ldots,n+1)}$ empirically exhibits factor
$$
\lambda = \frac{n!}{\lcm(1,2,\ldots,n)},
$$
(see sequence \cite[A025527]{OEIS}), and $q_{\red}\neq B(n)$.  
Thus, $\F_n$ and $P_n^{(2,3,\ldots,n+1)}$ define distinct toric varieties in terms of the classification of \cite{C02}.  
Moreover, $\F_n$ appears to be the only known weighted projective space with Eulerian $h^\ast$-polynomial.  

A second way to see that $\F_n$ and $P_n^{(2,3,\ldots,n+1)}$ define distinct toric varieties is to note that $\F_n$ is not \emph{self-dual} for $n\geq 3$; that is, for $n\geq 3$, the polar body of $\F_n$ is not $\F_n$ itself up to a translation and/or a rotation.  
On the other hand, \cite{HOT16} proves that $P_n^{(2,3,\ldots,n+1)}$ is a self-dual polytope.  
Thus, from a discrete geometric perspective, this allows us to see that $\F_n$ and $P_n^{(2,3,\ldots,n+1)}$ define distinct toric varieties.  
\end{remark}

\begin{remark}[Type B Eulerian polynomials]
\label{rmk: type B eulerian polynomials}
We also note that the \emph{Type B Eulerian polynomials} cannot arise from a numeral system since the sequence $a := (2^{n}\cdot n!)_{n=0}^\infty$ is a mixed radix system that is not reflexive (see Example~\ref{ex: reflexive and non-reflexive mixed radix systems}). 
\end{remark}

\subsection{The binary numbers and binomial coefficients}
\label{subsec: the binary numbers and binomial coefficients}
Using Proposition~\ref{prop: unique divisor system for reflexive mixed radix systems}, we can see that the binary numeral system $a = (a_n)_{n=0}^\infty = (2^n)_{n=0}^\infty$ is also reflexive (see Example~\ref{ex: reflexive and non-reflexive mixed radix systems}).  
Here, the $h^\ast$-polynomial of the resulting $n$-simplices are given by counting the number of $1$'s in the binary representation of the first $2^n$ nonnegative integers.  
In the following, we let $\supp_2(b)$ denote the number of nonzero digits in the binary representation 
$
b_2 := b_2(n-1)b_2(n-1)\cdots b_2(0)
$
of the integer $b$.

\begin{theorem}
\label{thm: binomial coefficients}
The binary numeral system $a = (a_n)_{n=0}^\infty = (2^n)_{n=0}^\infty$ admits the divisor system $d = (d_n)_{n=0}^\infty = (2^{n+1})_{n=0}^\infty$ for which the reflexive simplex $\Delta_{(1,q)}\subset\R^n$ with
$$
q:=\left( 
1,2,4,8,\ldots,2^{n-1}
\right)
$$
has $h^\ast$-polynomial
$$
h^\ast(\Delta_{(1,q)};z) = \sum_{b=0}^{2^n-1}z^{\supp_2(b)} = (1+z)^n.
$$
\end{theorem}

\begin{proof}
To prove the result we show that $\omega(b) = \supp_2(b)$ for all $b = 0,1,2,\ldots,2^n-1$ via induction on $n$.  
For the base case, we take $n=1$.  
By \cite[Theorem 2.5]{BDS16} we have that
$$
h^*(\Delta_{(1,q)};z) = z^{\omega(0)}+z^{\omega(1)}
$$
where
$
\omega(0) = 0 = \supp_2(0),
$
and
$
\omega(1) = 1 = \supp_2(1).
$
For the inductive step, suppose that $\omega(b) = \supp_2(b)$ for all $b = 0,1,2,\ldots,2^{n-1}$.  
Then by Proposition~\ref{prop: recursion for reflexive numeral systems} we have that 
$$
\omega(b) = \omega(b^\prime) + b_2(n-1) - \left\lfloor\frac{b}{d_{n-1}}\right\rfloor = \omega(b^\prime) + b_2(n-1),
$$
where the last equality holds since $0\leq b< 2^n$.  
Since 
$$
\supp_2(b) = \supp_2(b^\prime) + b_2(n-1),
$$
the result follows by the inductive hypothesis.  
Finally, the fact that $h^\ast(\Delta_{(1,q)}) = (1+z)^n$ follows from \cite[Theorem 1]{F85} and the fact that any $(0,1)$-string of length $n$ is a valid binary representation of a nonnegative integer less than $2^n$.  
\end{proof}

Similar to the factoradics, the $q$-vector in Theorem~\ref{thm: binomial coefficients} can be viewed as the coefficients of a ``max-descent" polynomial for binary strings of length $n$.  Namely, $q_i$ is the number of binary strings $\eta_{n-1}\eta_{n-2}\cdots\eta_0$ with right-most nonzero digit $\eta_i$.  
Analogously, Theorem~\ref{thm: binomial coefficients} provides a geometric transformation between these two generating polynomials for counting binary strings in terms of their nonzero entries.  

Proposition~\ref{prop: unique divisor system for reflexive mixed radix systems} implies that is the only reflexive base-$r$ numeral system for $r\geq 2$ is the binary system.  
Thus, there does not exist a base-$r$ generalization of Theorem~\ref{thm: binomial coefficients} that results in simplices with symmetric $h^\ast$-polynomials.  
On the other hand, there is a generalization that preserves other desirable properties of the $h^\ast$-polynomial $(1+x)^n$, including real-rootedness and unimodality.  
This is the focus of the next section.

\section{The Positional Base-$r$ Numeral Systems}
\label{sec: the positional base r numeral systems}
For $r\geq2$, consider  the base-$r$ numeral system, $a := (r^n)_{n=0}^\infty$.  
Just as in the base-$2$ case, we denote the base-$r$ representation of an integer $b\in\Z_{\geq0}$ by
$$
b_r := b_r(n-1)b_r(n-2)\cdots b_r(0).  
$$
We will now study a generalization of the $n$-simplices from Theorem~\ref{thm: binomial coefficients} for $r\geq 2$ whose $h^\ast$-polynomials preserve many of the nice properties of the $r=2$ case, including real-rootedness and unimodality.  
Our generalization is motivated as follows:

In order to simplify the formula for $\omega(b)$ to the desired state in the proofs of Theorem~\ref{thm: eulerian polynomials} and Theorem~\ref{thm: binomial coefficients} respectively, we required the identities
$$
1+\sum_{k=0}^{n-1} k\cdot k! = n!
\quad
\mbox{ and}
\quad 
1+\sum_{k=0}^{n-1}2^k = 2^n.
$$
Notice that these identities are different than the one requested in Definition~\ref{def: reflexive numeral system} $(2)$ to certify reflexivity of a numeral system.  
In fact, it follows from \cite[Theorem 2]{F85} that any mixed radix system $a = (a_n)_{n=0}^\infty$ with sequence of radices $c = (c_n)_{n=1}^\infty$ satisfies the identity
$$
1+\sum_{k=0}^{n-1}(c_{k+1}-1)a_k = a_n.
$$  
In the case of base-$r$ numeral systems, this identity yields a natural generalization of Theorem~\ref{thm: binomial coefficients}.  
For two integers $r\geq2$ and $n\geq1$, we define the \emph{base-$r$ $n$-simplex} to be the $n$-simplex $\B_{(r,n)}:=\Delta_{(1,q)}\subset\R^n$ for
$$
q:=\left( 
(r-1),(r-1)r,(r-1)r^2,\ldots,(r-1)r^{n-1}
\right).
$$
In the following, we show that, while symmetry of $h^\ast(\B_{(r,n)};z)$ does not hold for $r>2$, many of the nice properties of $h^\ast(\B_{(2,n)};z)$ carry over to this more general family.  
In Subsection~\ref{subsec: a descent-like statistic} we prove that $h^\ast(\B_{(r,n)};z)$ admits a combinatorial interpretation in terms of a descent-like statistic applied to the nonzero digits of the base-$r$ representations of the nonnegative integers.  
Then, in Subsection~\ref{subsec: real-rootedness and unimodality}, we prove that $h^\ast(\B_{(r,n)};z)$ is real-rooted and unimodal for all $r\geq 2$ and $n\geq 1$.  

\subsection{A descent-like statistic}
\label{subsec: a descent-like statistic}
We now define a descent-like statistic on the base-$r$ representations of nonnegative integers that we then use to give a combinatorial interpretation of the $h^\ast$-polynomial of $\B_{(r,n)}$ for $r\geq2$ and $n\geq1$.  
Given two indices $i\geq j$ and an integer $b\in\Z_{\geq0}$, we can think of the integer quantity
$
b_r(i)-b_r(j)
$
as the \emph{height} of the index $i$ ``above" the index $j$ in the string $b_r$.  
Of course, a negative height simply means we think of $i$ as ``below" $j$ in $b_r$.  
We define the \emph{(average weighted) height} of an index $i>0$ in $b_r$ to be
$$
\awheight(i) :=\frac{1}{i}\sum_{j=0}^{i-1}(b_r(i)-b_r(j))r^j, 
\,
\mbox{ and}
\,\,\,\,
\awheight(0) := 
\begin{cases}
0	&	\mbox{if $b_r(0) = 0$},	\\
1	&	\mbox{if $b_r(0) \neq 0$}.	\\
\end{cases}
$$
In a sense, this statistic measures the height of $i$ above the remaining substring of $b_r$ where the value of the height of an index closer to $i$ (in absolute value) is higher.  
When $\awheight(i)$ is nonnegative, we can think of $i$ as being at least as high as the remaining portion of the string, and so we say that $i$ is a \emph{nonascent} of $b_r$ if 
$
0\leq\awheight(i).
$
Define the \emph{support of $b$} to be the set
$$
\Supp_r(b):= \{i\in\Z_{\geq0}: b_r(i) \neq 0\} 
\quad 
\mbox{ and let}
\quad
\supp_r(b) := \left|\Supp_r(b)\right|.
$$
We then consider the collection of indices
$$
\Nasc_r(b) := \{i\in\Supp_r(b) : 0\leq\awheight(i)\},
$$
and we let $\nasc_r(b) := \left|\Nasc_r(b)\right|.$

\begin{example}[Computing $\nasc_r(b)$ for base-$4$ numerals]
\label{ex: descent-like statistic}
The base-$4$ numeral system is $a = (4^n)_{n=0}^\infty$.  
The numeral representations for the first 64 nonnegative integers are the strings of length three $b_4(2)b_4(1)b_4(0)$ in which each term $b_4(i)$ can assume values $0,1,2$, or $3$.  
For example, the number $b = 19$ has base-$4$ representation $103$, and so the support of $b$ is $\Supp_4(b) = \{0,2\}$.  
To compute $\nasc_4(b)$ we must compute the (averaged weighted) height of the indices $0$ and $2$.  
Since $b_4(0) = 3$ then $\awheight(0) = 1$, and 
$$
\awheight(2) = \frac{1}{2}\left((1-0)\cdot4^1+(1-3)\cdot4^0\right) = 1.
$$
Thus, $\nasc_4(b) = 2$.  
Table~\ref{table: height statistics} presents this statistic for a couple more integers.  
\begin{table}
\begin{center}
\begin{tabular}{| c | c | c | c | c |}\hline
$b$		& 	base-$4$ numeral	&	$\Supp_4(b)$	&	$\awheight$ of numerals in $\Supp_4(b)$		&	$\nasc_4(b)$	\\\hline
19		&		$103$		&	$\{0,2\}$		&	$\awheight(0) = 1$						&		$2$		\\
		&					&				&	$\awheight(2) = 1$						&				\\\hline
22		&		$112$		&	$\{0,1,2\}$		&	$\awheight(0) = 1$						&		$1$		\\
		&					&				&	$\awheight(1) = -1$						&				\\
		&					&				&	$\awheight(2) = -\dfrac{1}{2}$				&				\\\hline
31		&		$133$		&	$\{0,1,2\}$		&	$\awheight(0) = 1$						&		$2$		\\
		&					&				&	$\awheight(1) = 0$						&				\\
		&					&				&	$\awheight(2) = -5$						&				\\\hline
\end{tabular}
\medskip
\caption{The statistic $\nasc_r(b)$ for some integers $b$ with respect to the base-$4$ numeral system $a = (4^n)_{n=0}^\infty$ and $n=3$.}
\label{table: height statistics} 
\end{center}
\end{table}
Given the statistic $\nasc_4(b)$ for all integers $0\leq b<64$ we can compute 
$$
h^\ast(\B_{(4,3)};z) = \sum_{b=0}^{63}  z^{\nasc_4(b)} = 1+19z+34z^2+10x^3.
$$
The following theorem shows that this formula generalizes to the base-$r$ $n$-simplices for all $n\geq 1$ and $r\geq2$.  
\end{example}

\begin{theorem}
\label{thm: hstar polynomials for base r simplices}
For two integers $r\geq2$ and $n\geq1$ the base-$r$ $n$-simplex $\B_{(r,n)}$ has $h^\ast$-polynomial
$$
h^\ast(\B_{(r,n)};z) = \sum_{b=0}^{r^n-1}z^{\nasc_r(b)}.
$$
\end{theorem}

\begin{proof}
By \cite[Theorem 2.5]{BDS16}, it suffices to show for $0\leq b<r^n$ that
$
\omega(b) = \nasc_r(b).
$
We prove this fact via induction on $n$.  Notice first that 
$$
\omega(b) = b - \sum_{k = 1}^n\left\lfloor\frac{(r-1)b}{r^i}\right\rfloor.
$$
For the base, take $n=1$, and notice that $\omega(b) = \left\lceil\frac{b}{r}\right\rceil$, and so $\omega(b) = 0$ where $b=0$, and $\omega(b) = 1$ for $1\leq b\leq r-1$.  
Suppose now that the result holds for $n-1$.  
Letting $b^\prime := b - b_r(n-1)r^{n-1}$, we then observe that 
\begin{equation*}
\begin{split}
\omega(b) 
&= b - \sum_{k=1}^n\left\lfloor\frac{(r-1)b}{r^i}\right\rfloor,	\\
&= b - \sum_{k=1}^n\left\lfloor\frac{(r-1)(b_r(n-1)r^{n-1}+b^\prime)}{r^i}\right\rfloor,	\\
&= b - \sum_{k=1}^n\left\lfloor\frac{(r-1)b_r(n-1)r^{n-1}}{r^i}\right\rfloor - \sum_{k=1}^n\left\lfloor\frac{(r-1)b^\prime}{r^i}\right\rfloor - \left\lfloor\frac{(r-1)b}{r^n}\right\rfloor,	\\
&= \omega(b^\prime)+b_r(n-1)r^{n-1} - b_r(n-1)\left(\sum_{k=0}^{n-2}r^k\right) - \left\lfloor\frac{(r-1)b}{r^n}\right\rfloor,	\\
&= \omega(b^\prime)+ \left\lceil\frac{rb_r(n-1)r^{n-1}-(r-1)b^\prime}{r^n}\right\rceil,	\\
\end{split}
\end{equation*}
Notice that $rb_r(n-1)r^{n-1}-(r-1)b^\prime>0$ if and only if $b_r(n-1)\neq0$ and 
\begin{equation*}
\begin{split}
b^\prime &< b_r(n-1)\left(\frac{r^{n-1}}{r-1}\right),		\\
\sum_{j=0}^{n-2}b_r(j)r^j &< \sum_{j-0}^{n-2}b_r(n-1)r^j + \frac{b_r(n-1)}{r-1},		\\
 - \frac{b_r(n-1)}{(r-1)(n-1)}&< \frac{1}{n-1}\sum_{j-0}^{n-2}(b_r(n-1)-b_r(j))r^j.		\\
\end{split}
\end{equation*}
Since $1\leq b_r(n-1)\leq r-1$, this last inequality is equivalent to $n-1$ being a nonascent of $b_r$, which completes the proof.
\end{proof}


\subsection{Real-rootedness and unimodality}
\label{subsec: real-rootedness and unimodality}
To prove real-rootedness of the $h^\ast$-polynomial of $\B_{(r,n)}$, we will use the well-developed theory of interlacing polynomials. 
Let $f,g\in\R[z]$ be nonzero, real-rooted polynomials, and let $d:=\deg(f)$, and $c:=\deg(g)$ denote the degree of $f$ and $g$, respectively.  
Suppose $\alpha_d\leq\cdots\leq\alpha_2\leq\alpha_1$ and $\beta_c\leq\cdots\leq\beta_2\leq\beta_1$ are the roots of $f$ and $g$, respectively.  
We say that $g$ \emph{interlaces} $f$, written $g\preceq f$, if either $d = c$ and 
$$
\beta_d\leq\alpha_d\leq\cdots\leq\beta_2\leq\alpha_2\leq\beta_1\leq\alpha_1,
$$
or $d = c+1$ and
$$
\alpha_{d+1}\leq\beta_d\leq\alpha_d\leq\cdots\leq\beta_2\leq\alpha_2\leq\beta_1\leq\alpha_1.  
$$
If all inequalities are strict, we say that $g$ \emph{strictly interlaces} $f$ and we write $g\prec f$.  

A sequence $F_m:=(f_i)_{i=1}^m$ of real-rooted polynomials is called \emph{(strictly) interlacing} if $f_i$ (strictly) interlaces $f_j$ for all $1\leq i < j \leq m$.  
Let $\mathcal{F}_m^+$ denote the space of all interlacing sequences $F_m$ for which $f_i$ has only nonnegative coefficients for all $1\leq i \leq n$.  
In \cite{B15}, Br\"and\'en characterized when a matrix $G = (G_{i,j}(z))_{i,i=1}^m$ of polynomials maps $\mathcal{F}_m^+$ to $\mathcal{F}_m^+$.  
We say that such a map \emph{preserves interlacing}.  
It preserves \emph{strict} interlacing if it further maps strictly interlacing sequences to strictly interlacing sequences.  
In the following we use such polynomial maps to prove the real-rootedness of $h^\ast(\B_{(r,n)};z)$ for $r\geq2$, $n\geq1$.  

For $r\geq2$ and $n\geq1$ define the univariate polynomial
$$
f_{(r,n)}:=(1+z+z^2+\cdots+z^{r-1})^n.
$$
As noted in \cite{J16}, for every $r\geq1$ and $f\in\R[z]$ there are uniquely determined $f^{(0)},\ldots,f^{(r-2)}\in\R[z]$ such that 
$$
f(z) = f^{(0)}(z^{r-1})+zf^{(1)}(z^{r-1})+\cdots+z^{r-2}f^{(r-2)}(z^{r-1}).  
$$
So, for $\ell = 0,1,\ldots,r-2$, we consider the operator
$$
\,^{\langle r-1,\ell\rangle}:\R[z]\longrightarrow\R[z]
\quad
\mbox{ where}
\quad
\,^{\langle r-1,\ell\rangle}:f\longrightarrow f^{(\ell)}.
$$
The following theorem gives a second interpretation of the $h^\ast$-polynomial of $\B_{(r,n)}$, now in terms of the polynomials $f_{(r,n)}^{\langle r-1,\ell\rangle}$.  

\begin{theorem}
\label{thm: algebraic expression of hstar polynomials}
For two integers $r\geq2$ and $n\geq1$ the base-$r$ $n$-simplex $\B_{(r,n)}$ has $h^\ast$-polynomial
$$
h^\ast(\B_{(r,n)};z) = f_{(r,n)}^{\langle r-1,0\rangle}+z\sum_{\ell=1}^{r-2}f_{(r,n)}^{\langle r-1,\ell\rangle}.
$$
\end{theorem}

\begin{proof}
To prove this result, we prove a slightly stronger statement.  
We will show, via induction on $n$, that 
$$
\sum_{b=0}^{1+r+r^2+\cdots+r^{n-1}}z^{\omega(b)} = f_{(r,n)}^{\langle r-1,0\rangle},
$$
and for each $i\in[r-2]$,
$$
\sum_{b = 1+i(1+r+r^2+\cdots+r^{n-1})}^{(i+1)(1+r+r^2+\cdots+r^{n-1})}z^{\omega(b)} = zf_{(r,n)}^{\langle r-1,r-\ell-1\rangle}.
$$
For the base case, we let $n=1$, and so we must verify that 
$$
z^{\omega(0)}+z^{\omega(1)} = f_{(r,1)}^{\langle r-1,0\rangle},
$$
and for each $i\in[r-2]$,
$$
z^{\omega(i+1)} = zf_{(r,1)}^{\langle r-1,r-\ell-1\rangle}.  
$$
Notice first that since $n=1$, then $\omega(b) = \left\lceil\frac{b}{r}\right\rceil$ for all $b = 0,1,\ldots, r-1$, and so 
$$
\omega(b) = 
\begin{cases}
0 	&	\mbox{if $b=0$,}	\\
1	&	\mbox{if $b\in[r-1]$.}	\\
\end{cases}
$$
The base case follows immediately from the fact that 
$
f_{(r,1)} = 1+r+r^2+\cdots+r^{n-1}.
$

As for the inductive step, we begin by partitioning the sequence of numbers 
$$
\mathbb{B} = (\mathbb{B}_j)_{j=0}^{r^n-1} := (0,1,2,\ldots,r^n-1)
$$
into $r$ consecutive sequences
$$
B_i = (B_{i,j})_{j=0}^{r^{n-1}-1} := (\mathbb{B}_j)_{j=ir^{n-1}}^{(i+1)r^{n-1}-1}
$$
for $i = 0,1,\ldots,r-1$.  
Notice that if $b\in B_i$ then $b_r(n-1) = i$.  
Even more, the number $b$ has the base-$r$ representation
$
b_r(n-1)b_r(n-2)\cdots b_r(1)b_r(0)
$
being the $b^{th}$ sequence of $n$ digits $0,1,\ldots,r-1$ in lexicographic ordering.  
It then follows that for all $i = 1,2,\ldots,r-1$
\begin{equation}
\label{eqn: 1}
\omega(B_{i,j}) =
\begin{cases}
\omega(B_{0,j})+1 	&	\mbox{if $j\leq1+r+r^2+\cdots+r^{n-1}$,}	\\
\omega(B_{0,j})		&	\mbox{if $j>1+r+r^2+\cdots+r^{n-1}$.}	\\
\end{cases}
\end{equation}
This is because the base-$r$ representation of $b = i(1+r+r^2+\cdots+r^{n-1})$ is $b_r = ii\cdots i$.  
Combining the observation in equation (\ref{eqn: 1}) with the inductive hypothesis, we then have that
\begin{equation*}
\begin{split}
\sum_{b=0}^{1+r+r^2+\cdots+r^{n-1}}z^{\omega(b)} 
&= \sum_{b=0}^{r^{n-1}-1}z^{\omega(b)}+\sum_{b=r^{n-1}}^{1+r+r^2+\cdots+r^{n-1}}z^{\omega(b)},	\\
&= h^\ast(\B_{(r,n-1)};z)+zf_{(r,n-1)}^{\langle r-1,0\rangle},		\\
&= f_{(r,n-1)}^{\langle r-1,0\rangle}+z\sum_{\ell = 1}^{r-2}f_{(r,n-1)}^{\langle r-1,\ell\rangle}+zf_{(r,n-1)}^{\langle r-1,0\rangle},	\\
&=f_{(r,n)}^{\langle r-1,0\rangle},\\
\end{split}
\end{equation*}
which proves the first part of the claim.  

For the second part of the claim, we just want to see that for $i =1,2,\ldots,r-2$
$$
\sum_{b = B_{i,2+r+r^2+\cdots+r^{n-2}}}^{B_{i,r^{n-1}-1}}z^{\omega(b)}+\sum_{b=B_{i+1,0}}^{B_{i+1,1+r+r^2+\cdots+r^{n-2}}}z^{\omega(b)} =
zf_{(r,n)}^{\langle r-1,r-\ell-1\rangle}.
$$
However, by the inductive hypothesis and equation (\ref{eqn: 1}) it follows that 
\begin{equation*}
\begin{split}
\sum_{b = B_{i,2+r+r^2+\cdots+r^{n-2}}}^{B_{i,r^{n-1}-1}}z^{\omega(b)}+\sum_{b=B_{i+1,0}}^{B_{i+1,1+r+r^2+\cdots+r^{n-2}}}z^{\omega(b)} 
&=\sum_{\ell = i}^{r-2}f_{(r,n-1)}^{\langle r-1,\ell\rangle}+z\sum_{\ell=0}^if_{(r,n-1)}^{\langle r-1,\ell\rangle},	\\
&= f_{(r,n)}^{\langle r-1,r-\ell-1\rangle}.\\
\end{split}
\end{equation*}
Thus, the claim holds for all $r\geq 2$ and $n\geq 1$.  
The desired expression for $h^\ast(\B_{(r,n)};z)$ then follows immediately from \cite[Theorem 2.5]{BDS16}.  
\end{proof}

\begin{example}[The $h^\ast$-polynomial of a base-$4$ simplex]
\label{ex: hstar-polynomial of a base-4 simplex}
Example~\ref{ex: descent-like statistic} presents the $h^\ast$-polynomial of the base-$4$ $3$-simplex $\B_{(4,3)}$ in terms of the average weighted height statistic.  
We can recompute this polynomial using the formula proved in Theorem~\ref{thm: algebraic expression of hstar polynomials}.  
If we expand the polynomial $f_{(4,3)}$ and write it diagrammatically as
\begin{equation*}
\begin{split}
f_{(4,3)} &= (1+z+z^2+z^3)^3,\\
	&= 1+3z+6z^2+10z^3+12z^4+12z^5+10z^6+6z^7+3z^8+z^9,\\
	&= 1\hspace{49pt}+10z^{1\cdot3}\hspace{58pt}+10z^{2\cdot3}\hspace{47pt}+z^{3\cdot3} \\
	&	\hspace{20pt}+3z^{0\cdot3+1} \hspace{36pt} + 12z^{1\cdot3+1} \hspace{47pt} + 6z^{2\cdot3+1} \\
	&	\hspace{42pt}+6z^{0\cdot3+2} \hspace{46pt} + 12z^{1\cdot3+2} \hspace{42pt} + 3z^{2\cdot3 +2}, \\
\end{split}
\end{equation*}
then we see by the decomposed presentation of $f_{(4,3)}$ in the third equality that 
\begin{equation*}
\begin{split}
f_{(4,3)}^{\langle 3,0\rangle} &= 1+10z+10z^2+z^3,\\
f_{(4,3)}^{\langle 3,1\rangle} &= 3+12z+6z^2, \mbox{ and}\\
f_{(4,3)}^{\langle 3,2\rangle} &= 6+12z+3z^2.\\
\end{split}
\end{equation*}
Then, by Theorem~\ref{thm: algebraic expression of hstar polynomials}, we know that
\begin{equation}
\label{eqn: palindromic decomposition}
\begin{split}
h^\ast(\B_{(4,3)};z) &= (1+10z+10z^2+z^3)+z(9+24z+9z^2),\\
	&= 1+19z+34z^2+10z^3,\\
\end{split}
\end{equation}
and thus we recover the $h^\ast$-polynomial originally computed in Example~\ref{ex: descent-like statistic}.
\end{example}

\begin{remark}[The symmetric decomposition of an $h^\ast$-polynomial]
\label{rmk: palindromic decompositions}
The first line of equation~(\ref{eqn: palindromic decomposition}) in Example~\ref{ex: hstar-polynomial of a base-4 simplex} highlights a more general phenomenon.  
It is a well-known result that if $P$ is a lattice polytope containing an interior lattice point, then there is a unique decomposition of $h^\ast(P;z)$ as
$$
h^\ast(P;z) = a(z) + zb(z),
$$
where $a(z) = z^da\left(\frac{1}{z}\right)$ and $b(z) = z^{d-1}b\left(\frac{1}{z}\right)$.
Moreover, these polynomials admit a nice combinatorial interpretation as
\begin{equation*}
\begin{split}
a(z) &= \sum_{\Delta\in T}h(\link_T(\Delta);z)B_\Delta(z),	\mbox{ and}\\
b(z) &= \frac{1}{z}\sum_{\Delta\in T}h(\link(\Delta);z)B_{\conv(\Delta,0)}(z),\\
\end{split}
\end{equation*}
where $T$ is any triangulation of the boundary of $P$, $h(\link_T(\Delta);z)$ is the $h$-polynomial of the link of the simplex $\Delta$ in $T$, and $B_S(z)$ is the box polynomial of a simplex $S$.  
A summary of these various definitions and a proof of this decomposition of $h^\ast(P;z)$ is provided in \cite[Chapter 10, Theorem 10.5]{BR07}.  

For the base-$r$ $n$-simplex $\B_{(r,n)}$, it follows from Theorem~\ref{thm: algebraic expression of hstar polynomials} that 
$$
a(z) = f_{(r,n)}^{\langle r-1,0\rangle} 
\qquad 
\mbox{and} 
\qquad 
b(z) = \sum_{\ell=1}^{r-2}f_{(r,n)}^{\langle r-1,\ell\rangle}.  
$$
In Theorem~\ref{thm: real-rootedness}, we will use the formulation of $h^\ast(\B_{(r,n)};z)$ given in Theorem~\ref{thm: algebraic expression of hstar polynomials} to prove that $h^\ast(\B_{(r,n)};z)$ is real-rooted and unimodal.  
In fact, it follows along the way, that the symmetric polynomials $a(z)$ and $b(z)$ for $\B_{(r,n)}$ have these properties as well.  
To the best of the author's knowledge, this is the first known proof of real-rootedness of a (non-symmetric) $h^\ast$-polynomial that comes by way of proving its symmetric decomposition consists of real-rooted polynomials as well.  
\end{remark}

We also note that Theorem~\ref{thm: algebraic expression of hstar polynomials} provides us with a second combinatorial interpretation of the coefficients of $h^\ast(\B_{(r,n)};z)$.  
Given a subset $S\subset\Z_{>0}$ and two integers $t,m\in\Z_{>0}$, we let $\comp_t(m;S)$ denote the number of compositions of $m$ of length $t$ with parts in $S$.  
Since for all $k\in\Z_{\geq0}$ the coefficient of $z^k$ in $f_{(r,n)}$ is
$$
[z^k].f_{(r,n)} = \comp_n\left(n+k;[r]\right),
$$
then we have the following corollary to Theorem~\ref{thm: algebraic expression of hstar polynomials}.
\begin{corollary}
\label{cor: compositions expression}
For integers $r\geq2$ and $n\geq1$ 
$$
[z^k].h^\ast(\B_{(r,n)};z) = \comp_n\left(n+k;[r]\right)+\sum_{\ell=1}^{r-2}\comp_n\left(n+(k-1)(r-1)+\ell;[r]\right),
$$
for each $k = 0,1,\ldots,n$.  
\end{corollary}

We now use Theorem~\ref{thm: algebraic expression of hstar polynomials} to verify that $h^\ast(\B_{(r,n)};z)$ is real-rooted.  
To do so, we first prove that a useful polynomial map $G$ preserves (strict) interlacing.  
\begin{lemma}
\label{lem: interlacing family}
The polynomial map
$$
G := 
\begin{pmatrix}
z+1		&	1	&	1		&	\cdots	& 	1		\\
z		&	z+1	&	1		&			& 	\vdots	\\
z		&	z	&	z+1		&	\ddots	& 	1		\\
\vdots	&		&	\ddots	&	\ddots	& 	1		\\
z		&	z	&	\cdots	&	z		& 	z+1		\\
\end{pmatrix}
\in\R[z]^{(r-1)\times(r-1)}
$$
preserves strict interlacing.  
\end{lemma}

\begin{proof}
In \cite[Theorem 7.8.5]{B15} Br\"and\'en gives a complete characterization of all such matrices.  
Applying this characterization, it suffices to prove that each of the five $2\times 2$ matrices 
$$
\begin{pmatrix}
1	&	1	\\
1	&	1	\\
\end{pmatrix},
\quad
\begin{pmatrix}
z	&	z	\\
z	&	z	\\
\end{pmatrix},
\quad
\begin{pmatrix}
z+1	&	1	\\
z	&	z+1	\\
\end{pmatrix},
\quad
\begin{pmatrix}
1	&	1	\\
z+1	&	1	\\
\end{pmatrix},
\quad
\mbox{and }
\begin{pmatrix}
z	&	z+1	\\
z	&	z	\\
\end{pmatrix}
$$
preserve interlacing and nonnegativity.  
This follows from a series of results in \cite[Section 3.11]{F08} proven by Fisk.  
In particular, the result follows for the first two matrices by applying \cite[Lemma 3.71]{F08}, for the third matrix by \cite[Lemma 3.79]{F08}, and for the fourth matrix by  \cite[Lemma 3.83(1)]{F08}.  
Finally, the fifth matrix is seen to preserve interlacing and nonnegativity by factoring it as 
$$
\begin{pmatrix}
z	&	z+1	\\
z	&	z	\\
\end{pmatrix}
=
\begin{pmatrix}
1	&	1	\\
0	&	1	\\
\end{pmatrix}
\begin{pmatrix}
0	&	1	\\
z	&	z	\\
\end{pmatrix},
$$
and applying \cite[Lemma 3.71]{F08} to each of these factors.  
\end{proof}

Using Lemma~\ref{lem: interlacing family}, we can now prove our main result of this subsection.  
\begin{theorem}
\label{thm: real-rootedness}
For two integers $r\geq2$ and $n\geq 1$, the $h^\ast$-polynomial of the base-$r$ $n$-simplex $\B_{(r,n)}$ is real-rooted and thus unimodal.  
\end{theorem}

\begin{proof}
By Theorem~\ref{thm: algebraic expression of hstar polynomials} we know that the $h^\ast$-polynomial of $\B_{(r,n)}$ is expressible as 
$$
h^\ast(\B_{(r,n)};z) = f_{(r,n)}^{\langle r-1,0\rangle}+z\sum_{\ell=1}^{r-2}f_{(r,n)}^{\langle r-1,\ell\rangle}.  
$$
Notice next that
\begin{equation}
\label{eqn: recursion for f polynomial}
f_{(r,n)} = (1+z+z^2+\cdots+z^{r-1})f_{(r,n-1)}.
\end{equation}
For each index $k$, let $a_k:=[z^k].f_{(r,n-1)}$ and $b_k:=[z^k].f_{(r,n)}$.  
Then recall that we can write each $k$ uniquely as $k = i(r-1)+j$ for integers $i$ and $0\leq j<r-1$.  
It follows that
$$
[z^i].f_{(r,n-1)}^{\langle r-1,j\rangle} = a_{i(r-1)+j},
\qquad 
\mbox{and similarly}
\qquad
[z^i].f_{(r,n)}^{\langle r-1,j\rangle} = b_{i(r-1)+j}.
$$
Therefore, for each $\ell = 0,1,\ldots, r-2$, it follows from equation~(\ref{eqn: recursion for f polynomial}) that 
$$
[z^i].f_{(r,n)}^{\langle r-1,\ell\rangle}  = b_{i(r-1)+\ell} = \sum_{j=0}^\ell a_{i(r-1)+(\ell-j)}+\sum_{j=\ell}^{r-2}a_{(i-1)(r-1)+(r-2+\ell-j)}.
$$
For an example of this computation, we refer the reader to Example~\ref{ex: hstar-polynomial of a base-4 simplex}.  
More concisely, this expression for $[z^i].f_{(r,n)}^{\langle r-1,\ell\rangle}$ for each index $i$ is equivalent to saying that the vector of polynomials
$$
\begin{pmatrix}
f_{(r,n)}^{\langle r-1,r-2\rangle}	&	 \cdots & f_{(r,n)}^{\langle r-1,1\rangle}	& 	f_{(r,n)}^{\langle r-1,0\rangle} \\
\end{pmatrix}^T
$$
is produced by multiplying the vector of polynomials
$$
\begin{pmatrix}
f_{(r,n-1)}^{\langle r-1,r-2\rangle}	&	 \cdots & f_{(r,n-1)}^{\langle r-1,1\rangle}	& 	f_{(r,n-1)}^{\langle r-1,0\rangle} \\
\end{pmatrix}^T
$$
on the left by the $(r-1)\times(r-1)$ matrix of polynomials $G$ in Lemma~\ref{lem: interlacing family}.  
Thus, by way of induction, Lemma~\ref{lem: interlacing family} implies that 
$f_{(r,n)}^{\langle r-1,r-2\rangle}\prec \cdots \prec f_{(r,n)}^{\langle r-1,1\rangle}\prec f_{(r,n)}^{\langle r-1,0\rangle}$ 
 is a sequence of strictly mutually interlacing and nonnegative polynomials.  
Moreover, it is well-known that the polynomial map 
$$
H := 
\begin{pmatrix}
1		&	1	&	1		&	\cdots	& 	1		\\
z		&	1	&	1		&			& 	\vdots	\\
z		&	z	&	1		&	\ddots	& 	1		\\
\vdots	&		&	\ddots	&	\ddots	& 	1		\\
z		&	z	&	\cdots	&	z		& 	1		\\
\end{pmatrix}
\in\R[z]^{(r-2)\times(r-2)}
$$
also preserves interlacing.  
For instance, proofs of this fact can be found in \cite[Example 3.73]{F08}, \cite[Corollary 7.8.7]{B15}, and \cite[Proposition 2.2]{J16}.  
Applying this map once to the vector of polynomials
$$
\begin{pmatrix}
f_{(r,n)}^{\langle r-1,r-2\rangle}	&	 \cdots & f_{(r,n)}^{\langle r-1,1\rangle}	& 	f_{(r,n)}^{\langle r-1,0\rangle} \\
\end{pmatrix}^T
$$
produces a vector of polynomials
$$
\begin{pmatrix}
g_{r-2}	&	 \cdots & g_1	& g_0, \\
\end{pmatrix}^T
$$
and it follows that $g_{r-2}\prec \cdots\prec g_1\prec g_{0}$ is a sequence of strictly mutually interlacing polynomials with the property that $g_0 = h^\ast(\B_{(r,n)};z)$.  
The result then follows.  
\end{proof}

\begin{remark}
\label{rmk: new family of interlacing polynomials}
In order to prove Theorem~\ref{thm: real-rootedness} we used Lemma~\ref{lem: interlacing family} to first show that 
$$
\left(f_{(r,n)}^{\langle r-1,r-\ell-1\rangle}\right)_{\ell=1}^{r-1}
$$
is a strictly interlacing sequence.  
Other important $h$-polynomials have been shown to be real-rooted using a closely related construction.  
In particular, in order to verify a conjecture of \cite{BS10}, Jochemko shows in \cite{J16} that the sequence 
$$
\left(f_{(r,n)}^{\langle r,r-\ell\rangle}\right)_{\ell=1}^{r}
$$
is strictly interlacing.  
Similarly, in \cite{L16} and \cite{Z16} Leander and Zhang independently showed that 
$$
\left(f_{(r,n)}^{\langle r+1,r-\ell+1\rangle}\right)_{\ell=1}^{r+1}
$$
is a strictly interlacing sequence in order to prove that the $r^{th}$ edgewise and cluster subdivisions of the simplex have real-rooted local $h$-polynomials.  
Each of these strictly interlacing sequences constitutes a distinct family of real-rooted polynomials, and collectively they represent the growing prevalence of decompositions of the polynomial $f_{(r,n)}$ in unimodality questions for $h$-polynomials.  
\end{remark}

\section{A Closing Remark}
\label{sec: a closing remark}
To conclude our discussions, we remark that two natural classes of simplices associated to numeral systems have been introduced and analyzed in this note.  
In Section~\ref{sec: reflexive numeral systems}, we searched for numeral systems admitting divisors systems that allowed us to construct a $q$-vector yielding a reflexive $n$-simplex with normalized volume the $n^{th}$ place value for all $n\geq0$.  
In the identified examples, we saw that these numeral systems yield combinatorial interpretations of the associated $h^\ast$-polynomials that are closely related to interpretations for the associated $q$-vectors.  
Furthermore, these examples all exhibited the desirable distributional properties implied by real-rootedness.  
To produce more examples of this nature, the (seemingly difficult) problem is to identify a divisor system for some numeral system $a$, and then study the geometry of the simplex whose $q$-vector is given by identity $(2)$ in Definition~\ref{def: reflexive numeral system}.

On the other hand, in Section~\ref{sec: the positional base r numeral systems}, we used a natural choice of $q$-vector for any mixed radix numeral system to generalize the geometry associated to the binary numeral system in Theorem~\ref{thm: binomial coefficients}.  
The result is a family of nonreflexive simplices associated to the base-$r$ numeral systems for $r\geq2$.  
We observed that while symmetry of the $h^\ast$-polynomial is lost in this generalization, real-rootedness is preserved.  
This suggests a possible extension to a larger family of simplices with real-rooted $h^\ast$-polynomials, namely, the simplices associated to mixed radix systems via an analogous choice of $q$-vector.  
Amongst such mixed radix simplices, the simplices $\B_{(r,n)}$ correspond exactly to mixed radix systems in which all radices are taken to be equal, and our proof of real-rootedness relies heavily on this fact.  
Thus, new techniques may be necessary to address real-rootedness for this larger family.  
On a related note, computations suggest that the simplices $\B_{(r,n)}$ are also \emph{Ehrhart positive}.  
This curious observation further serves to promote the study of simplices for numeral systems from the combinatorial perspective.  
\bigskip

\noindent
{\bf Acknowledgements}. 
The author was supported by an NSF Mathematical Sciences Postdoctoral Research Fellowship (DMS - 1606407). 
He would also like to thank Petter Br\"and\'en and Benjamin Braun for helpful discussions on the project.

%
%

\begin{thebibliography}{99}

\bibitem{BR07}
	M. Beck and S. Robins. 
	\emph{Computing the continuous discretely.} 
	Springer Science+ Business Media, LLC, 2007.

\bibitem{BS10}
	M. Beck and A. Stapledon. 
	\emph{On the log-concavity of Hilbert series of Veronese subrings and Ehrhart series.}
	Mathematische Zeitschrift 264.1 (2010): 195-207.

\bibitem{BD16}
	B. Braun and R. Davis. 
	\emph{Ehrhart series, unimodality, and integrally closed reflexive polytopes.}
	Annals of Combinatorics 20.4 (2016): 705-717.

\bibitem{BDS16} 
	B. Braun, R. Davis, and L. Solus.
	\emph{Detecting the integer decomposition property and Ehrhart unimodality in reflexive simplices.}
	Submitted to Discrete and Computational Geometry.
	Preprint available at \url{https://arxiv.org/abs/1608.01614} (2016).

\bibitem{B15}
	P. Br\"and\'en. 
	\emph{Unimodality, log-concavity, real-rootedness and beyond.}
	Handbook of Enumerative Combinatorics (2015): 437-483.
	
\bibitem{CXGP91}
	P. Candelas, C. Xenia, P. S. Green, and L. Parkes.
	\emph{A pair of Calabi-Yau manifolds as an exactly soluble superconformal theory.}
	Nuclear Physics B 359.1 (1991): 21-74.
	
\bibitem{C02}
	H. Conrads. 
	\emph{Weighted projective spaces and reflexive simplices.} 
	manuscripta mathematica 107.2 (2002): 215-227.
	
\bibitem{CK99}
	D. A. Cox and S. Katz. 
	\emph{Mirror symmetry and algebraic geometry.}
	No. 68. American Mathematical Soc., 1999.
	
\bibitem{E62}
	E. Ehrhart.  
	\emph{Sur les polyh\`{e}dres rationnels homoth\`{e}tiques \`{a} $n$ dimensions}.  
	C. R. Acad Sci. Paris, 254:616-618, 1962.

	
\bibitem{F08}
	S. Fisk. 	
	\emph{Polynomials, roots, and interlacing.} 
	Preprint available at \url{https://arxiv.org/abs/math/0612833} (2008).

\bibitem{F85}
	A. S. Fraenkel. 
	\emph{Systems of numeration.}
	The American Mathematical Monthly, Vol. 92, No. 2 (Feb., 1985), pp. 105-114.

\bibitem{H92}
	T. Hibi. 
	\emph{Note dual polytopes of rational convex polytopes.}
	Combinatorica 12.2 (1992).
	
\bibitem{HOT16}
	T. Hibi, M. Olsen, and A. Tsuchiya.
	\emph{Self-dual reflexive simplices with Eulerian polynomials.}
	Graphs and Combinatorics DOI 10.1007/s00373-017-1781-8, 2017.  

\bibitem{J16} 
	K. Jochemko. 
	\emph{On the real-rootedness of the Veronese construction for rational formal power series.}
	Preprint available at \url{https://arxiv.org/abs/1602.09139} (2016).

\bibitem{K98}
	D. E. Knuth. 
	\emph{The art of computer programming: sorting and searching.} 
	Vol. 3. Pearson Education, 1998.
	
\bibitem{L16}
	M. Leander. 
	\emph{Compatible polynomials and edgewise subdivisions.}
	Preprint available at \url{https://arxiv.org/abs/1605.05287} (2016).
	
\bibitem{N07}
	B. Nill. 
	\emph{Volume and lattice points of reflexive simplices.}
	Discrete \& Computational Geometry 37.2 (2007): 301-320.
	
\bibitem{OEIS}
	N.~J.~Sloane. 
	\emph{The On-Line Encyclopedia of Integer Sequences}. (2003).

	
\bibitem{S80}
	R. P. Stanley. 
	\emph{Decompositions of rational convex polytopes.}
	Annals of discrete mathematics 6 (1980): 333-342.
	
\bibitem{Z16}
	P. B. Zhang. 
	\emph{On the Real-rootedness of the Local $h$-polynomials of Edgewise Subdivisions of Simplexes.}
	 Preprint available at \url{https://arxiv.org/abs/1605.02298} (2016).

\end{thebibliography}
\end{document}